\pdfoutput=1
\documentclass[a4paper]{article}
\usepackage[english]{babel}
\usepackage{a4wide}
\usepackage{amsmath,amssymb,amsthm}
\usepackage{graphicx}
\usepackage{enumerate}
\usepackage{ifpdf}
\usepackage{cite}

\ifpdf
\usepackage[bookmarks=false,
pdfstartview=FitH,linkbordercolor={0.5 1 1},
citebordercolor={0.5 1 0.5},unicode,
pagebackref,hyperindex]{hyperref}
\else
\fi

\usepackage[scr=boondox,scrscaled=1.05]{mathalfa}
\usepackage{tikz}
\usetikzlibrary{shapes,arrows,positioning,fit,calc}

\newtheorem{theorem}{Theorem}[section]
\newtheorem{lemma}[theorem]{Lemma}
\newtheorem{proposition}{Proposition}
\newtheorem*{problem}{Problem}

\newtheoremstyle{myremark}
{6pt}
{6pt}
{\rmfamily}
{}
{\bfseries}
{.}
{.5em}
{}

\newtheorem{definition}[theorem]{Definition}
\newtheorem{remark}{Remark}
\newtheorem{example}{Example}

\DeclareMathOperator*{\co}{co}
\newcommand{\setA}{\mathscr{A}}
\newcommand{\setB}{\mathscr{B}}
\newcommand{\setI}{\mathscr{I}}
\newcommand{\setH}{\mathscr{H}}
\newcommand{\setM}{\mathscr{M}}
\newcommand{\systemA}{\boldsymbol{A}}
\newcommand{\systemAB}{\boldsymbol{AB}}
\newcommand{\plantA}{\mathbscr{A}} 
\newcommand{\controllerB}{\mathbscr{B}}

\title{Minimax joint spectral radius and stabilizability of discrete-time linear switching control systems}

\author{Victor Kozyakin\thanks{Kharkevich Institute for Information Transmission Problems, Russian Academy of Sciences, Bolshoj Karetny lane 19, Moscow 127051, Russia, e-mail: kozyakin@iitp.ru\newline\indent Kotel'nikov Institute of Radio-engineering and Electronics, Russian Academy of Sciences, Mokhovaya 11-7, Moscow 125009, Russia}}

\date{}
\begin{document}
\maketitle

\begin{abstract}
To estimate the growth rate of matrix products $A_{n}\cdots A_{1}$ with factors from some set of matrices $\setA$, such numeric quantities as the joint spectral radius $\rho(\setA)$ and the lower spectral radius $\Check{\rho}(\setA)$ are traditionally used. The first of these quantities characterizes the maximum growth rate of the norms of the corresponding products, while the second one characterizes the minimal growth rate. In the theory of discrete-time linear switching systems, the inequality $\rho(\setA)<1$ serves as a criterion for the stability of a system, and the inequality $\Check{\rho}(\setA)<1 $ as a criterion for stabilizability.

For matrix products $A_{n}B_{n}\cdots A_{1}B_{1}$ with factors $A_{i}\in\setA$ and $B_{i}\in\setB$, where $\setA$ and $\setB$ are some sets of matrices, we introduce the quantities $\mu(\setA,\setB)$ and $\eta(\setA,\setB)$, called the lower and upper minimax joint spectral radius of the pair $\{\setA,\setB\}$, respectively, which characterize the maximum growth rate of the matrix products $A_{n}B_{n}\cdots A_{1}B_{1}$ over all sets of matrices $A_{i}\in\setA$ and the minimal growth rate over all sets of matrices $B_{i}\in\setB$. In this sense, the minimax joint spectral radii can be considered as generalizations of both the joint and lower spectral radii. As an application of the minimax joint spectral radii, it is shown how these quantities can be used to analyze the stabilizability of discrete-time linear switching control systems in the presence of uncontrolled external disturbances of the plant.
\medskip

\medskip

\end{abstract}

\section{Introduction}

Various applied and theoretical problems of computational mathematics, control theory, coding theory, combinatorics, etc. lead to the necessity to know the growth/decrease rate of products of $(N\times N)$-matrices $A_{n}\cdots A_{1}$ with factors from some set of matrices $\setA$, see, e.g.,~\cite{Jungers:09,Theys:PhD05}, and also the bibliography in~\cite{Koz:IITP13}. To estimate the growth rate of the corresponding matrix products such numeric characteristics of the set of matrices $\setA$ are conventionally used as the joint spectral radius~\cite{RotaStr:IM60}
\begin{equation}\label{E-JSR}
\rho(\setA)= \lim_{n\to\infty}\sup\left\{\|A_{n}\cdots A_{1}\|^{\frac{1}{n}}:~A_{i}\in\setA\right\}
\end{equation}
and the lower spectral radius~\cite{Gurv:LAA95}
\begin{equation}\label{E-LSR}
\Check{\rho}(\setA)= \lim_{n\to\infty}\inf\left\{\|A_{n}\cdots A_{1}\|^{\frac{1}{n}}:~A_{i}\in\setA\right\},
\end{equation}
also called the joint spectral subradius. The limits in~\eqref{E-JSR} and~\eqref{E-LSR} always exist and do not depend on the norm $\|\cdot\|$ on the space of matrices of dimension $N\times N$; the corresponding proofs with historical comments can be found, e.g., in~\cite{Jungers:09,Theys:PhD05}.

The concepts of the joint and lower spectral radius arose in the second half of the 20th century and to date several hundred publications have been devoted to their investigation, see, e.g., the bibliography in~\cite{Jungers:09,Koz:IITP13}. One of the areas in which the application of joint and lower spectral radii is most natural and productive is the theory of linear switching systems with discrete time. In particular, the inequality $\rho(\setA)<1$ turns out to be a stability criterion for a discrete-time linear switching system, and the inequality $\Check{\rho}(\setA)<1 $ is a criterion for stabilizability.

Despite the fact that the joint and lower spectral radii are determined by `almost identical' equalities~\eqref{E-JSR} and~\eqref{E-LSR}, their properties vary significantly. It suffices to mention only the fact that the joint spectral radius $\rho(\setA)$ in the natural sense depends continuously on the set $\setA$, while the lower spectral radius $\Check{\rho}(\setA)$ in general is not a continuous function of $\setA$, see strict formulations, e.g., in~\cite{BochiMor:PLMS15,Jungers:LAA12}. Moreover, a number of properties of the joint and lower spectral radii, which in the final formulation look the same, are proved with the help of completely different approaches.

What has been said above provokes a natural desire, in the author's opinion, to introduce a certain characteristic of matrix products, which would unite the concepts of both joint and lower spectral radii. To realize this idea, we consider matrix products $A_{n}B_{n}\cdots A_{1}B_{1}$ with factors $A_{i}\in\setA$, $B_{i}\in\setB$ from two different sets of matrices $\setA$ and $\setB$ for which the quantities
\begin{align*}
\mu(\setA,\setB)= \lim_{n\to\infty}\max_{A_{i}\in\setA}\min_{B_{i}\in\setB} \|A_{n}B_{n}\cdots A_{1}B_{1}\|^{\frac{1}{n}},\\
\eta(\setA,\setB)= \lim_{n\to\infty} \min_{B_{i}\in\setB}\max_{A_{i}\in\setA} \|A_{n}B_{n}\cdots A_{1}B_{1}\|^{\frac{1}{n}},
\end{align*}
are defined, which we call further the lower and upper minimax joint spectral radius of the pair $\{\setA,\setB\}$. Both the quantities $\mu(\setA,\setB)$ and $\eta(\setA,\setB)$ characterize the maximum growth rate of the matrix products $A_{n}B_{n}\cdots A_{1}B_{1}$ over all sets of matrices  $A_{i}\in\setA$ and the minimal growth rate over all sets of matrices $B_{i}\in\setB$.

Outline the content of the work. This section presents the research motivation. In Section~\ref{S:classic}, the basic facts of the theory of joint/lower spectral radius are recalled and the relationship of these concepts to the problems of stability and stabilizability of discrete-time linear switching systems in which there is no control over the parameters of the system is explained. In Section~\ref{S:MJSR-stabilize}, the question is discussed of how the problem of the stabilizability of linear switching systems with discrete time changes in the case when a controller is added to such systems that allows controlling the parameters of the system. In Section~\ref{S:main}, the principal concepts in this paper of the lower and upper minimax joint spectral radii for two sets of matrices $\setA$ and $\setB$ are introduced. Here it is also shown (Theorems~\ref{T:pointwizestab} and~\ref{T:strongstab}) that the inequalities $\mu(\setA,\setB)<1$ and $\eta(\setA,\setB)<1$ are the criteria for various variants of stabilization of control systems in the presence of uncontrollable external disturbances of the plant. Section~\ref{S:MMprop} is devoted to a more detailed analysis of some properties of the minimax joint spectral radii. In particular, in it one class of matrices $\setA$ and $\setB$ is described, for which the lower and upper minimax joint spectral radii (and also some other similar quantities) coincide (Theorem~\ref{T:Hsets}). In Section~\ref{S:questions}, a brief discussion of the results is conducted and some open questions are formulated. Finally, Section~\ref{S:aux} contains proofs of all the statements of the paper.

\section{Stability/stabilizability of uncontrolled linear switching systems}\label{S:classic}

In this section, we recall the control-theoretic motivation for attracting the concepts of joint and lower spectral radius for the analysis of the problem of stability and stabilizability of (uncontrolled) linear switching systems.

Let us consider the discrete-time switching dynamical system $\systemA$ with delay in feedback, consisting of the plant $\plantA$, shown in Fig.~\ref{F:1}, whose output is additively affected by the external perturbation $f$.

\begin{figure}[htbp!]
\centering
 \tikzstyle{block}=[draw,fill=black!15,rectangle,minimum height=2em,minimum width=2.5em]
 \tikzstyle{sum}=[draw,circle,inner sep=0pt,minimum size=0.75em]
 \tikzstyle{dot}=[draw,fill=black,circle,inner sep=0pt,minimum size=0.1em]
 \tikzstyle{empty}=[coordinate]
 \begin{tikzpicture}[auto, node distance=3em and 2em,node font=\small,>=latex']
 \node [empty, name=Input] {};
 \node [sum, below of = Input] (Sum) {\tiny$+$};
 \node [block, left = of Sum] (Plant) {$\plantA$};
 \node [block, below of = Plant] (Delay) {$\boldsymbol{z^{-1}}$};
 \node [dot, right = of Sum] (Dot) {};
 \node [empty, right = of Dot] (Dummy0) {};
 \node [empty, below of = Dot] (Dummy1) {};
 \node [empty, right = of Dummy0,xshift=-0.8em] (Output) {};
 \node [empty, left = of Plant,xshift=-3em] (Dummy2) {};
 \draw [->] (Input)+(0,1em) -- (Input) node[right,yshift=-0.5em]{$f(n)$} -- (Sum);
 \draw [->] (Plant) -- (Sum);
 \draw [->] (Sum) -- (Dot) -- node {$x(n)$} (Dummy0) -- (Output);
 \draw [->] (Dot) |- (Delay);
 \draw [->] (Delay) -|  (Dummy2) -- node {$x(n-1)$} (Plant);
 \node[draw,dashed,inner sep=0.5em,label={[shift={(-1.5em,-1.5em)}]north east:$\systemA$},fit= (Dummy0) (Delay) (Dummy2) (Input)] {};
\end{tikzpicture}
\caption{Discrete-time linear switching system}\label{F:1}
\end{figure}
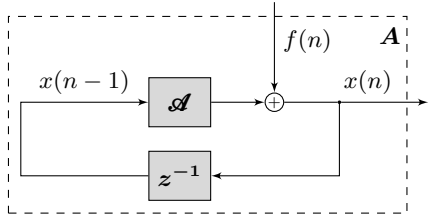

We assume that for each time instance $n=1,2,\ldots$ the output $x_{\text{out}}$ and the input $x_{\text{in}}$ of the plant $\plantA$ are connected by the linear equation
\begin{equation}\label{E:plantA}
  x_{\text{out}}=A_{n}x_{\text{in}},\qquad x_{\text{in}},x_{\text{out}}\in\mathbb{R}^{N},
\end{equation}
where $A_{n}$ is a matrix of dimension $N\times{}N$ that takes values in some finite set of matrices $\setA$.

The sequence of matrices $\{A_{n}\}$, depending on the context, can either be determined by external disturbances, or formed in some way in order to give the whole system some properties. The function $f(n)$ represents additive external actions on the state vector $x$ of the system. The block $\boldsymbol{z^{-1}}$ is the delay element per unit of time (one clock cycle). In this case the dynamics of the system under consideration is described by the inhomogeneous equation
\[
x(n)=A_{n}x(n-1)+f(n),\qquad n=1,2\ldots\,,
\]
in which the variables $x(n)$ and $f(n)$ are assumed to be column-vectors of dimension~$N$.

In the case when there are no additive external actions of $f$, that is, $f(n)\equiv0$, the dynamics of the system is described by the homogeneous equation
\begin{equation}\label{E:mainA}
  x(n)=A_{n}x(n-1),\qquad n=1,2,\ldots\,.
\end{equation}

\begin{definition}\label{D:asystab}
A system $\systemA$ with zero input $f$, governed by equation~\eqref{E:mainA}, is called \emph{asymptotically stable} in the class of all matrices $\setA$ if
\begin{equation}\label{E:Antoinfty}
x(n)=A_{n}\cdots A_{1}x(0)\to 0\quad\text{as}\quad n\to\infty
\end{equation}
for any sequence of matrices $\{A_{n}\in\setA\}$ and any initial condition $x(0)$.
\end{definition}

As is known~\cite{BerWang:LAA92,DaubLag:LAA92,Gurv:LAA95,Koz:AiT90:10:e}, the convergence to zero of all solutions of equation~\eqref{E:mainA} in a class of matrices $\setA$ implies stronger property of exponential convergence to zero of each sequence $\{X_{n}\}$ of matrix products $X_{n}=A_{n}\cdots A_{1}$, i.e. the existence of constants $C>0$ and $\lambda\in(0,1)$ (independent of the matrix factors $A_{1},\ldots, A_{n}$) such that $\|A_{n}\cdots A_{1}\|\le C \lambda^{n}$, where $\|\cdot\|$ is some norm on the space of $(N\times N)$-matrices. The latter property in turn implies the fulfillment of the inequality $\rho(\setA)<1$, where $\rho(\setA)$ is the joint spectral radius of the set of matrices $\setA$ determined by equality~\eqref{E-JSR}. On the other hand, the fulfillment of the inequality $\rho(\setA)<1$  in an obvious way implies the convergence to zero of each sequence of matrices $X_{n}=A_{n}\cdots A_{1}$ with cofactors from $\setA$. The following assertion follows from this.

\begin{proposition}\label{P:1}
A system $\systemA$ with zero input $f$, governed by equation~\eqref{E:mainA}, is asymptotically stable in the class of matrices  $\setA$ if and only if $\rho(\setA)<1$.
\end{proposition}

\begin{definition}\label{D:pw-stab}
A system $\systemA$ with zero input $f$ is called \emph{pointwise stabilizable} in the class of all matrices $\setA$ if for each initial condition $x(0)$ there is a sequence of matrices $\{A_{n}\in\setA\}$, for which the convergence~\eqref{E:Antoinfty} holds.
\end{definition}

\begin{definition}\label{D:unistab}
A system $\systemA$  with zero input $f$ is called \emph{uniformly stabilizable} or simply \emph{stabilizable} in the class of all matrices $\setA$ if there exists a sequence of matrices $\{A_{n}\in\setA\}$ such that the convergence~\eqref{E:Antoinfty} holds for each initial condition $x(0)$.
\end{definition}

Obviously, the uniform stabilizability of a system $\systemA$ with zero input $f$ is equivalent to the existence of a sequence of matrices $\{A_{n}\in\setA\}$, such that the matrix products $A_{n}\cdots A_{1}$ are normwise convergent:
\[
\|A_{n}\cdots A_{1}\|\to 0\quad\text{as}\quad n\to\infty.
\]

In the literature, various terminology is used to denote concepts equivalent to pointwise or uniformly stabilizability. For example, in a number of works, instead of the term stabilizability, there is used a broader term controllability that goes back to R.~Kalman, see, e.g.,~\cite{JunMas:SIAMJCO17,LinAnt:IEEETAC09,SunGe05}.
In~\cite{SU:SIAMJMAA94}, for the concepts of pointwise or uniform stabilizability, the terms pointwise or uniform convergence of matrix products with matrices from $\setA$ are used. Uniform convergence of matrices implies their pointwise convergence, while the converse is not true.

\begin{example}[\cite{Stanford:SIAMJCO79,SU:SIAMJMAA94}]\label{Ex:1}
The products of matrices from the set
\[
\setA = \left\{\begin{bmatrix}
  \frac{1}{2}&0\\0&\vphantom{\frac{1}{2}}2
\end{bmatrix},~
\begin{bmatrix}
  \frac{\sqrt{3}}{2}&\frac{1}{2}\\-\frac{1}{2}&\frac{\sqrt{3}}{2}
\end{bmatrix}\right\},
\]
converge pointwise, but are not convergent uniformly.
\end{example}

To characterize the stabilizability, it is convenient to use the lower spectral radius $\Check{\rho}(\setA) $, defined by equality~\eqref{E-LSR}. In particular, the following assertion holds.

\begin{proposition}\label{P:2}
A system $\systemA$ with zero input $f$, governed by equation~\eqref{E:mainA}, is uniformly stabilizable if and only if $\Check{\rho}(\setA)<1$.
\end{proposition}

The sufficiency of the condition $\Check{\rho}(\setA)<1$ for stabilizability follows directly from formula~\eqref{E-LSR}. And as shown in~~\cite[Proposition~1]{JunMas:SIAMJCO17}, \cite{SU:SIAMJMAA94}, \cite[Theorem~3.9]{SunGe05}, the stabilizability of the system $\systemA$, governed by equation~\eqref{E:mainA}, implies the inequality $\Check{\rho}(\setA)<1$.

Thus, Propositions~\ref{P:1} and~\ref{P:2} demonstrate that the joint and lower spectral radii are convenient analytical tools for analyzing the stability and stabilizability of (uncontrolled) linear switching systems. Unfortunately, the calculation of both the joint and the lower spectral radius is a complex problem, and only in exceptional cases it is possible to describe the classes of matrices for which these characteristics can be calculated in an `explicit formula' form, see, e.g., the bibliography in~\cite{Jungers:09,Koz:IITP13}.

\section{Stabilizability of controlled linear switching systems}\label{S:MJSR-stabilize}

Let us turn to a more realistic discrete-time control system $\systemAB$, which includes not only the plant $\plantA$, but also the controller $\controllerB$, see Fig.~\ref{F:3}.

\begin{figure}[htbp!]
\centering
 \tikzstyle{block}=[draw,fill=black!15,rectangle,minimum height=2em,minimum width=2.5em]
 \tikzstyle{sum}=[draw,circle,inner sep=0pt,minimum size=0.75em]
 \tikzstyle{dot}=[draw,fill=black,circle,inner sep=0pt,minimum size=0.1em]
 \tikzstyle{empty}=[coordinate]
 \begin{tikzpicture}[auto, node distance=3em and 2em,node font=\small,>=latex']
 \node [empty, name=Input] {};
 \node [sum, below of = Input] (Sum) {\tiny$+$};
 \node [block, left = of Sum] (Plant) {$\plantA$};
 \node [block, left = of Plant,xshift=-3em] (Controller) {$\controllerB$};
 \node (MidPC) at ($(Plant)!0.5!(Controller)$) {};
 \node [block, below of = MidPC] (Delay) {$\boldsymbol{z^{-1}}$};
 \node [dot, right = of Sum] (Dot) {};
 \node [empty, below of = Dot] (Dummy1) {};
 \node [empty, right = of Dot] (Dummy0) {};
 \node [empty, right = of Dummy0,xshift=-0.8em] (Output) {};
 \node [empty, left = of Controller,xshift=-3em] (Dummy2) {};
 \node [empty, below of = Dummy2] (Dummy3) {};
 \draw [->] (Input)+(0,1em)  -- (Input) node[right,yshift=-0.5em] {$f(n)$} -- (Sum);
 \draw [->] (Plant) -- (Sum);
 \draw [->] (Sum) -- (Dot) -- node {$x(n)$} (Dummy0) -- (Output);
 \draw [->] (Dot) |- (Delay);
 \draw [->] (Delay) --  (Dummy3) --  (Dummy2) -- node {$x(n-1)$} (Controller);
 \draw [->] (Controller) -- node {$u(n-1)$} (Plant);
 \node[draw,dashed,inner sep=0.5em,label={[shift={(-2.5em,-1.5em)}]north east:$\systemAB$},fit= (Dummy0) (Delay) (Dummy2) (Input)] {};

\end{tikzpicture}
\caption{Control system consisting of plant~$\plantA$ and controller~$\controllerB$}\label{F:3}
\end{figure}

Concerning the plant $\plantA$, the same assumptions will be made as in the previous section, namely, we will assume that for each $n=1,2,\ldots$ the output $x_{\text{out}} $ of the plant $\plantA$ is linked with its input $x_{\text{in}}$ by the linear equation
\[
x_{\text{out}}=A_{n}x_{\text{in}},\qquad x_{\text{in}}\in\mathbb{R}^{M},~x_{\text{out}}\in\mathbb{R}^{N},
\]
where $A_{n}$ is an $(M\times{}N)$-matrix taking values in some finite set of matrices $\setA$. The difference from assumptions~\eqref{E:plantA} that were superimposed on the plant $\plantA$ in Section~\ref{S:classic} is that in this case the dimensions of the input and output of the plant $\plantA$ do not need to be the same.

The controller $\controllerB$ will also be assumed to be functioning for each $n=1,2,\ldots$ in accordance with the linear equation
\[
u_{\text{out}}=B_{n}u_{\text{in}},\qquad u_{\text{out}}\in\mathbb{R}^{M},~u_{\text{in}}\in\mathbb{R}^{N},
\]
in which the matrix $B_{n}$ of dimension $N\times{}M$ can be chosen from some finite set of matrices $\setB$ which can be treated as the set of all available controls.

In this context, the matrix sequence $\{A_{n}\}$ is determined by external (uncontrollable) disturbances of the plant, while the sequence of matrices $\{B_{n}\}$ represents the control actions of the controller, with which one can try to give desirable properties to the system under consideration. The function $f(n)$ represents additive input effects on the state vector of the system. The block $\boldsymbol{z^{-1}}$, as in Section~\ref{S:classic}, is the delay element per unit of time (one clock cycle). In this case, the dynamics of the system under consideration is governed by the equation
\[
x(n)=A_{n}B_{n}x(n-1)+f(n),\qquad n=1,2\ldots\,,
\]
in which $x(n)$ and $f(n)$ are column-vectors of dimension $N$.

Again, in order not to be distracted by nonessential details, we will only be interested in the stability and stabilizability of the zero solution of the system shown in Fig.~\ref{F:3}, in the case of zero input $f$, i.e. when $f(n)\equiv0$. The dynamics of such a system is governed by the equation
\begin{equation}\label{E:main}
x(n)=A_{n}B_{n}x(n-1),\qquad n=1,2\ldots\,.
\end{equation}

For a control system $\systemAB$ with zero input $f$, governed by equation~\eqref{E:main}, questions about (asymptotic) stability and stabilizability similar to those for $\systemA$ can be posed.

\begin{definition}\label{D:asystabAB}
A system $\systemAB$ with zero input $f$ is said to be \emph{asymptotically stable} in the class of all perturbations $\setA$ of the plant $\plantA$ and controls $\setB$ of the controller $\controllerB$ if
\begin{equation}\label{E:Antoinfty2}
x(n)=A_{n}B_{n}\cdots A_{1}B_{1}x(0)\to 0\quad\text{as}\quad n\to\infty,
\end{equation}
for any sequences of matrices $\{A_{n}\in\setA\}$,  $\{B_{n}\in\setB\}$ and any initial condition $x(0)$.
\end{definition}

We note, however, that the consideration of absolute stability for the system $\systemAB$ does not introduce anything new in comparison with the consideration of the system $\systemA$. Obviously, the system $\systemAB$ with zero input $f$, governed by equation~\eqref{E:main}, is asymptotically stable in the class of all perturbations of $\setA$ and controls $\setB$ if and only if the system $\systemA$ with zero input $f$ is asymptotically stable in the class of matrices
\[
\setA\setB:=\{AB:~A\in\setA,~B\in\setB\}.
\]
From this remark and Proposition~\ref{P:2} the following assertion follows.

\begin{theorem}\label{T:2}
A system $\systemAB$, governed by equation~\eqref{E:main}, is asymptotically stable in the class of all perturbations $\setA$ and controls $\setB$ if and only if $\rho(\setA\setB)<1$.
\end{theorem}

The question of the stabilizability of the system $\systemAB$ is less obvious. Let us consider only two variants of such stabilizability.

\begin{definition}\label{D:pathstab}
We say that a system $\systemAB$, governed by equation~\eqref{E:main}, is \emph {path-dependent stabilizable} in the class of all perturbations $\setA$ of the plant $\plantA$ by means of controls $\setB$ of the controller $\controllerB$ if for any sequence of matrices $\{A_{n}\in\setA\}$ (perturbations of the plant $\plantA$) there is a sequence of matrices $\{B_{n}\in\setB\}$ (control actions of the controller $\controllerB$) such that, for each initial condition $x(0)$, the convergence~\eqref{E:Antoinfty2} holds.
\end{definition}

\begin{definition}\label{D:perstab}
We say that the system $\systemAB$, governed by equation~\eqref{E:main}, is \emph {path-independent periodically stabilizable} if there exists a (universal) periodic sequence of matrices $\{B_{n}\in\setB\}$ (control actions of the controller $\controllerB$) such that, for each sequence of matrices $\{A_{n}\in\setA\}$ (perturbations of the plant $\plantA$) and each initial condition $x(0)$, the convergence~\eqref{E:Antoinfty2} holds.
\end{definition}

The question of the stabilizability of the system $\systemAB$ is close to the game-theoretic statements~\cite{ACDDHK:STACS15,BMRLL:AI16}, in which there are two players,  external influences and a controller, which alternately act on the plant, and first of which seeks to make the system as unstable as possible, and the second tries to stabilize it.

It is clear that path-independent periodically stabilizable systems are path-de\-pendent stabilizable. Moreover, in both definitions of stabilizability of the system $\systemAB$ the condition that the convergence~\eqref{E:Antoinfty2} holds for each initial condition $x(0)$ is equivalent to the condition
\[
\|A_{n}B_{n}\cdots A_{1}B_{1}\|\to0\quad\text{as}\quad n\to\infty.
\]

\begin{remark}\label{R:PDPS}
It would be possible to introduce pointwise analogs of the concepts of path-dependent and path-independent stabilizability, cf., e.g.,~\cite{JunMas:SIAMJCO17}, but this is not the purpose of the paper.
\end{remark}

\section{Minimax joint spectral radii}\label{S:main}

By analogy with the lower spectral radius characterizing the uniform stabilizability of the uncontrolled system $\systemA $, governed by equation~\eqref{E:mainA}, there naturally arises the desire to introduce some numeric values characterizing the path-independent and path-dependent stabilizability of the control system $\systemAB$, governed by equation~\eqref{E:main}. As candidates for such numeric values, we propose, respectively, the quantities
\begin{equation}\label{E:minmaxrad}
\mu(\setA,\setB)=\lim_{n\to\infty}\mu_{n}(\setA,\setB)^{\frac{1}{n}},\qquad
\eta(\setA,\setB)= \lim_{n\to\infty} \eta_{n}(\setA,\setB)^{\frac{1}{n}},
\end{equation}
where, for each $n=1,2,\ldots$\,,
\begin{equation}\label{E:defmnen}
\begin{aligned}
\mu_{n}(\setA,\setB)&=\max_{A_{i}\in\setA} \min_{B_{i}\in\setB}\|A_{n}B_{n}\cdots A_{1}B_{1}\|,\\
\eta_{n}(\setA,\setB)&= \min_{B_{i}\in\setB}\max_{A_{i}\in\setA} \|A_{n}B_{n}\cdots A_{1}B_{1}\|,
\end{aligned}
\end{equation}
and $\|\cdot\|$ is some norm on the space of matrices of dimension $N\times{}N$. Since the maximin of any function does not exceed its minimax, then
\[
\mu(\setA,\setB)\le\eta(\setA,\setB),
\]
which justifies the following definition.

\begin{definition}\label{D:minimaxSR}
Let $\{\setA,\setB\}$ be a pair of sets of matrices of dimension $N\times{}M $ and $M\times{}N$, respectively.  The quantity $\mu(\setA,\setB)$ will be called the \emph{lower}, and the quantity $\eta(\setA,\setB)$ the \emph{upper minimax joint spectral radius} of the pair $\{\setA,\setB\}$.
\end{definition}

The existence of the limits in~\eqref{E:minmaxrad} results from the following Lemma~\ref{L:MJSR-correctness}, the proof of which is given in Section~\ref{S:proof-MJSR-correctness}. Recall that a norm $\|\cdot\|$ in the space of matrices of dimension $N\times N$  is said to be \emph{submultiplicative} if $\|XY\|\le\|X\|\cdot\|Y\|$ for any matrices $X$ and $Y$. In particular, the matrix norm $\|\cdot\|$ is submultiplicative if it is generated by some vector norm, i.e. its value $\|A\| $ on the matrix $A$ is defined by the equality $\|A\|=\sup_{x\neq0}\frac{\|Ax\|}{\|x\|}$, where $\|x\|$ and $\|Ax\|$ are the norms of the corresponding vectors in $\mathbb{R}^{N}$.

\begin{lemma}\label{L:MJSR-correctness}
For any finite sets of matrices $\setA$ and $\setB$, the limits in~\eqref{E:minmaxrad} exist and do not depend on the norm $\|\cdot\|$. Moreover, if the norm $\|\cdot\|$ in~\eqref{E:defmnen} is submultiplicative, then
\begin{equation}\label{E:MJSR-inf}
\mu(\setA,\setB)=\inf_{n\ge0}\mu_{n}(\setA,\setB)^{\frac{1}{n}},\qquad
\eta(\setA,\setB)= \inf_{n\ge0} \eta_{n}(\setA,\setB)^{\frac{1}{n}}.
\end{equation}
\end{lemma}

It is natural to expect that in the general case $\mu(\setA,\setB)\neq\eta(\setA,\setB)$, which is confirmed by the following example.

\begin{example}\label{Ex:2}
Consider the sets $\setA=\{A_{1},A_{2}\}$ and $\setB=\{B_{1},B_{2}\}$, where
\[
A_{1}=\begin{bmatrix}
  2&0\\0&\frac{1}{2}
\end{bmatrix},\quad
A_{2}=\begin{bmatrix}
  3&0\\0&\frac{1}{3}
\end{bmatrix},\quad
B_{1}=\begin{bmatrix}
  \frac{1}{2}&0\\0&2
\end{bmatrix},\quad
B_{2}=\begin{bmatrix}
  \frac{1}{3}&0\\0&3
\end{bmatrix},
\]
then
\begin{equation}\label{E:mueat1}
\mu(\setA,\setB)= 1,\qquad\eta(\setA,\setB)> 1.
\end{equation}

To prove relations~\eqref{E:mueat1}, note first that by Lemma~\ref{L:MJSR-correctness} the quantities $\mu(\setA,\setB)$ and $\eta(\setA,\setB)$ do not depend on the choice of the norm. Therefore, in this example, for $\|\cdot\|$ we choose the matrix norm generated by the vector max-norm: $\|x\|=\max\left\{|x_{1}|,|x_{2}|\right\}$, where  $x=\{x_{1},x_{2}\}\in\mathbb{R}^{2}$.

We further note that for an arbitrary $(N\times N)$-matrix $A$ the inequalities
\[
\|A\|\ge\rho(A)\ge \left(\det A\right)^{\frac{1}{N}}
\]
hold, where $\rho(A)$ is the spectral radius of the matrix $A$. And since the determinant of each of the matrices $A_{1}$, $A_{2}$, $B_{1}$ and $B_{2}$ is $1$, then
\[
\|A_{i_{n}}B_{j_{n}}\cdots A_{i_{i}}B_{j_{1}}\|\ge \rho(A_{i_{n}}B_{j_{n}}\cdots A_{i_{1}}B_{i_{1}})\ge 1,
\]
for any choice of the indices $i_{k},j_{k}\in\{1,2\}$, $k=1,2\ldots,n$. Consequently, in our case
\begin{equation}\label{E:mueat0}
\mu(\setA,\setB)\ge 1,\qquad\eta(\setA,\setB)\ge 1.
\end{equation}

We now show that in fact stronger relations~\eqref{E:mueat1} are valid. Let us prove the first of them. Given an integer $n\ge 1$, consider an arbitrary collection of matrices $A_{i_{1}},\ldots, A_{i_{n}}$, where $i_{k}\in\{1,2\}$ for $k=1,2\ldots,n$. Then, for each $k=1,2\ldots,n$, the equality $A_{i_{k}}B_{i_{k}}=I$ takes place. So,  $A_{i_{n}}B_{i_{n}}\cdots A_{i_{i}}B_{i_{1}}=I$. In this case
\[
\min_{B_{i}\in\setB}\|A_{i_{n}}B_{n}\cdots A_{i_{1}}B_{1}\|\le \|A_{i_{n}}B_{i_{n}}\cdots A_{i_{1}}B_{i_{1}}\| =\|I\| = 1,
\]
and therefore, because of the arbitrariness of the matrices $A_{i_{1}},\ldots, A_{i_{n}}$,
\[
\max_{A_{i}\in\setA} \min_{B_{i}\in\setB}\|A_{n}B_{n}\cdots A_{1}B_{1}\|\le 1.
\]
Hence $\mu(\setA,\setB)\le 1$, which together with~\eqref{E:mueat0} leads to the first of relations~\eqref{E:mueat1}.

Let us prove the second of relations~\eqref{E:mueat1}. Again, given an integer $n\ge 1$, we consider an arbitrary set of matrices $B_{j_{1}},\ldots, B_{j_{n}}$, where $j_{k}\in\{1,2\}$ for $k=1,2\ldots,n$. Denote by $p$ the number of those indices $j_{k}$ for which $j_{k}=1$. Then for the remaining $n-p$ indices, the equality $j_ {k}=2$ will hold. Obviously, one of the numbers $p$ or $n-p$ is at least $\frac{n}{2}$. Therefore, without loss of generality, we can assume that $p\ge\frac{n}{2}$.

We now define the sequence $\{i_{k}\}$, $k=1,2\ldots,n$, setting $i_{k}\equiv 2$. In this case, in the matrix product $A_{i_{n}}B_{i_{n}}\cdots A_{i_{1}}B_{i_{1}}$  the factors $A_{i_{k}}B_{i_{k}}$ with sub-indices $k$ satisfying $j_{k}=1$ will coincide with the product
\begin{equation}\label{E:A2B1}
A_{2}B_{1}=
\begin{bmatrix}
  3&0\\0&\frac{1}{3}
\end{bmatrix}\cdot
\begin{bmatrix}
  \frac{1}{2}&0\\0&2
\end{bmatrix}=
\begin{bmatrix}
  \frac{3}{2}&0\\0&\frac{2}{3}
\end{bmatrix}.
\end{equation}
And the factors $A_{i_{k}}B_{i_{k}}$ with sub-indices $k$ for which $j_{k}=2 $ will coincide with the product $A_{2}B_{2}=I$. Therefore, the matrix product $A_{i_{n}}B_{i_{n}}\cdots A_{i_{1}}B_{i_{1}}$ can be computed explicitly:
\[
A_{i_{n}}B_{i_{n}}\cdots A_{i_{1}}B_{i_{1}}=\left(A_{2}B_{1}\right)^{p},
\]
whence in view of~\eqref{E:A2B1}
\[
\|A_{i_{n}}B_{i_{n}}\cdots A_{i_{1}}B_{i_{1}}\|= \left(\frac{3}{2}\right)^{p}\ge \left(\frac{3}{2}\right)^{\frac{n}{2}}.
\]
Then
\[
\max_{A_{i}\in\setA} \|A_{n}B_{j_{n}}\cdots A_{1}B_{j_{1}}\|\ge \|A_{i_{n}}B_{j_{n}}\cdots A_{i_{1}}B_{j_{1}}\| \ge \left(\frac{3}{2}\right)^{\frac{n}{2}},
\]
and because of the assumed arbitrariness of the matrices $B_{j_{1}},\ldots,B_{j_{n}}$, we get that
\[
\min_{B_{i}\in\setB}\max_{A_{i}\in\setA} \|A_{n}B_{n}\cdots A_{1}B_{1}\| \ge \left(\frac{3}{2}\right)^{\frac{n}{2}}.
\]
Therefore, $\eta(\setA,\setB)\ge \left(\frac{3}{2}\right)^{\frac{1}{2}}$, which leads to the second of relations~\eqref{E:mueat1}.
\end{example}

The following two theorems, the proofs of which are given in Section~\ref{S:proofstab}, are the main ones in the present paper. They confirm that the minimax joint spectral radii can actually act as characteristics of the stabilizability of the system $\systemAB$.

\begin{theorem}\label{T:pointwizestab}
A system $\systemAB$, governed by equation~\eqref{E:main}, is path-dependent stabilizable in the class of all perturbations $\setA$ and controls $\setB$ if and only if $\mu(\setA,\setB)<1$.
\end{theorem}

\begin{theorem}\label{T:strongstab}
A system $\systemAB$, governed by equation~\eqref{E:main}, is path-independent periodically stabilizable in the class of all perturbations $\setA$ and controls $\setB$ if and only if $\eta(\setA,\setB)<1$.
\end{theorem}

\section{Other minimax characteristics of matrix products}\label{S:MMprop}

According to~\cite{BerWang:LAA92}, in the definition~\eqref{E-JSR}  of the joint spectral radius $\rho(\setA)$, the norm $\|\cdot\|$ can be replaced by the spectral radius $\rho(\cdot)$ (with the simultaneous change of the limit to the upper limit):
\begin{equation}\label{E-GSR}
\rho(\setA)= \limsup_{n\to\infty}\sup\left\{\rho{(A_{n}\cdots A_{1})}^{\frac{1}{2}}:~A_{i}\in\setA\right\},
\end{equation}
the corresponding assertion is known as the Berger-Wang theorem~\cite{BerWang:LAA92}. Similarly, if in the definition~\eqref{E-LSR} we replace the norm $\|\cdot\|$ by the spectral radius $\rho(\cdot)$ and the limit by the lower limit, then we obtain another formula for the lower spectral radius
\begin{equation}\label{E-LSR1}
\Check{\rho}(\setA)= \liminf_{n\to\infty}\inf\left\{\rho{(A_{n}\cdots A_{1})}^{\frac{1}{2}}:~A_{i}\in\setA\right\}.
\end{equation}
For finite sets $\setA$ the validity of equality~\eqref{E-LSR1} was established in~\cite[Theorem~B1]{Gurv:LAA95}, and later for arbitrary sets $\setA$ a similar statement was proved in~\cite[Lemma~1.12]{Theys:PhD05} and~\cite[Theorem~1]{Czornik:LAA05}.

By analogy with formulas~\eqref{E-GSR} and~\eqref{E-LSR1} for the joint and lower spectral radius, we define the following minimax characteristics of matrix products:
\begin{alignat}{2}\label{E:minmaxhigh}
\Hat{\mu}(\setA,\setB)&=\limsup_{n\to\infty}\Bar{\mu}_{n}(\setA,\setB)^{\frac{1}{n}},&\quad
\Hat{\eta}(\setA,\setB)&= \limsup_{n\to\infty} \Bar{\eta}_{n}(\setA,\setB)^{\frac{1}{n}},
\intertext{and also}
\label{E:minmaxlow}
\Check{\mu}(\setA,\setB)&=\liminf_{n\to\infty}\Bar{\mu}_{n}(\setA,\setB)^{\frac{1}{n}},&\quad
\Check{\eta}(\setA,\setB)&= \liminf_{n\to\infty} \Bar{\eta}_{n}(\setA,\setB)^{\frac{1}{n}},
\end{alignat}
where, for each $n=1,2,\ldots$\,,
\begin{equation}\label{E:defbmbe}
\begin{aligned}
\Bar{\mu}_{n}(\setA,\setB)&=\max_{A_{i}\in\setA} \min_{B_{i}\in\setB}\rho(A_{n}B_{n}\cdots A_{1}B_{1}),\\
\Bar{\eta}_{n}(\setA,\setB)&= \min_{B_{i}\in\setB}\max_{A_{i}\in\setA} \rho(A_{n}B_{n}\cdots A_{1}B_{1}).
\end{aligned}\end{equation}

Obviously, along with the already introduced lower and upper minimax joint spectral radii $\mu(\setA,\setB)$ and $\eta(\setA,\setB)$, the quantities~\eqref{E:minmaxhigh} and~\eqref{E:minmaxlow} could also claim the role of numeric quantities that characterize the stabilizability of a controlled system $\systemAB$, governed by equation~\eqref{E:main}.

Since the spectral radius of a linear operator does not exceed its norm, and the quantities $\mu(\setA,\setB)$ and $\eta(\setA,\setB)$, as noted in Lemma~\ref{L:MJSR-correctness}, do not depend on the choice of the norm, then
\[
\mu(\setA,\setB)\ge \Hat{\mu}(\setA,\setB)\ge\Check{\mu}(\setA,\setB),\quad
\eta(\setA,\setB)\ge \Hat{\eta}(\setA,\setB)\ge\Check{\eta}(\setA,\setB).
\]
And since the maximin of any function does not exceed its minimax, then
\[
\Hat{\mu}(\setA,\setB)\le\Hat{\eta}(\setA,\setB),\quad\Check{\mu}(\setA,\setB)\le\Check{\eta}(\setA,\setB).
\]

From Example~\ref{Ex:2} it follows that in the general case $\mu(\setA,\setB)\neq\eta(\setA,\setB)$. And since all matrices in Example~\ref{Ex:2} are diagonal, then their norms coincide with the corresponding spectral radii. This implies that under the conditions of Example~\ref{Ex:2} the inequalities
\[
\Hat{\mu}(\setA,\setB)<\Hat{\eta}(\setA,\setB),\quad \Check{\mu}(\setA,\setB)<\Check{\eta}(\setA,\setB)
\]
are also satisfied. In this connection, the question arises of the existence of classes of matrices $\setA$ and $\setB$ for which the following equalities hold:
\begin{equation}\label{E:muetaeqalities}
\mu(\setA,\setB)=\eta(\setA,\setB)\quad\text{and/or}\quad \Hat{\mu}(\setA,\setB)=\Hat{\eta}(\setA,\setB),\quad  \Check{\mu}(\setA,\setB)=\Check{\eta}(\setA,\setB).
\end{equation}
At least one class of such matrices, introduced in~\cite{ACDDHK:STACS15,Koz:INFOPROC16:e,Koz:LMA17}, is described in the following Theorem~\ref{T:Hsets}. However, before proceeding to its formulation, we have to recall the necessary definitions and facts.

For the vectors  $x,y\in\mathbb{R}^{N}$, we will write $x\geqslant y$  (respectively, $x>y$) if the coordinates of the vector $x$ are not less than the corresponding coordinates of the vector $y$ (respectively, strictly greater than the corresponding coordinate of the vector~$y$). Similar notation will be applied to matrices. As usual, a vector or a matrix will be called nonnegative if all their elements are nonnegative, and positive if all their elements are positive.

We denote by $\setM(N,M)$ the set of real $(N\times M)$-matrices. Following~\cite{Koz:LAA16,Koz:LMA17}, we say that a set of positive matrices $\setA\subset\setM(N,M)$ is an \emph{$\setH$-set} (or \emph{hourglass set}) if for each pair $(\Tilde{A},u)$, where $\Tilde{A}$ is a matrix from the set $\setA$, and $u$ is a positive vector, the following statements hold:
\begin{quote}
\begin{itemize}
\item \emph{either $Au\geqslant\Tilde{A}u$ for all $A\in\setA$, or there exists a matrix $\bar{A}\in\setA$ such that $\bar{A}u\leqslant\Tilde{A}u$ and $\bar{A}u\neq\Tilde{A}u$};

\item \emph{either $Au\leqslant\Tilde{A}u$ for all $A\in\setA$, or there exists a matrix $\bar{A}\in\setA$ such that $\bar{A}u\geqslant\Tilde{A}u$ and $\bar{A}u\neq\Tilde{A}u$}.
\end{itemize}
\end{quote}
The set of all compact $\setH$-sets of positive matrices of dimention $N\times{}M$ will be denoted by $\setH(N,M)$.

\begin{example}\label{Ex:3}
We call a set of positive matrices $\setA=\{A_{1},A_{2},\ldots,A_{n}\}$ \emph{linearly ordered} if $0<A_{1}<A_{2}<\cdots<A_{n}$. Obviously, any linearly ordered set of positive matrices is an $\setH$-set of matrices. In particular, any set consisting of a single positive matrix is an $\setH$-set.
\end{example}

\begin{example}\label{Ex:4}
A less trivial and more interesting example of $\setH$-sets of matrices, as shown in~\cite[Lemma~4]{ACDDHK:STACS15} and~\cite[Lemma~1]{Koz:LAA16}, constitutes the totality of \emph{sets of positive matrices with independent row uncertainty}~\cite{BN:SIAMJMAA09} consisting of all matrices
\[
A=\begin{pmatrix}
a_{11}&a_{12}&\cdots&a_{1M}\\
a_{21}&a_{22}&\cdots&a_{2M}\\
\cdots&\cdots&\cdots&\cdots\\
a_{N1}&a_{N2}&\cdots&a_{NM}
\end{pmatrix},
\]
wherein each of the rows $a_{i} = (a_{i1}, a_{i2}, \ldots, a_{iM})$ of dimension $M$ is taken from some set of rows $\setA_{i}$, $i=1,2,\ldots,N$.

For example, if $N=2$  and $\setA_{1}=\{(a,b),(c,d)\}$, $\setA_{2}=\{(p,q),(r,s)\}$ then
\[
\setA=\left\{
\begin{bmatrix}
a&b\\p&q
\end{bmatrix},~
\begin{bmatrix}
c&d\\p&q
\end{bmatrix},~
\begin{bmatrix}
a&b\\r&s
\end{bmatrix},~
\begin{bmatrix}
c&d\\r&s
\end{bmatrix}\right\}
\]
will be the set of matrices with independent row uncertainty.
\end{example}

\begin{example}\label{Ex:5}
In~\cite[Example~3]{Koz:LMA17} a more general construction is described for constructing non-trivial $\setH$-sets of matrices, which uses the fact that the totality of all $\setH $-sets of matrices is algebraically closed~\cite[Theorem~2]{Koz:LAA16} with respect to the operations of Minkowski summation and multiplication over matrix sets.
\end{example}

\begin{theorem}\label{T:Hsets}
Let $\setA,\setB$ be compact $\setH$-sets of positive matrices of dimension $M\times N$ and $N\times M$, respectively. Then
\[
\mu(\setA,\setB)=\eta(\setA,\setB)= \Hat{\mu}(\setA,\setB)=\Hat{\eta}(\setA,\setB)= \Check{\mu}(\setA,\setB)=\Check{\eta}(\setA,\setB).
\]
\end{theorem}

The proof of Theorem~\ref{T:Hsets} is given in Section~\ref{S:proofTHsets}.

\section{Comments and open problems}\label{S:questions}

Let $\setA$ be a set of square matrices of dimension $N\times N$, and $\setI:=\{I\}$ be the one-element set of matrices consisting of the identity matrix of dimension $N\times{}N$. Then the following equalities are obvious:
\begin{equation}\label{E:LSReqs}
\begin{aligned}
\rho(\setA)=\mu(\setA,\setI) = \eta(\setA,\setI)=\Hat{\mu}(\setA,\setI)=\Hat{\eta}(\setA,\setI),\\
\Check{\rho}(\setA)=\mu(\setI,\setA)= \eta(\setI,\setA)=\Check{\mu}(\setI,\setA)=\Check{\eta}(\setI,\setA).
\end{aligned}
\end{equation}

\begin{remark}\label{Rem:continuity}
The joint spectral radius $\rho(\cdot)$ is in the natural sense continuous and even is a locally Lipschitz function of its argument, see the details and exact formulations in~\cite{HStr:LASP92,Jungers:09,Koz:LAA10,Wirth:LAA02}. At the same time, the lower spectral radius $\Check{\rho}(\cdot)$ in the general case is not a continuous function~\cite{BochiMor:PLMS15,BM:JAMS02,Jungers:09}. But due to~\eqref{E:LSReqs}
\[
\mu(\setI,\setA)=\eta(\setI,\setA)=\Check{\rho}(\setA),
\]
and therefore in the general case neither $\mu(\cdot,\cdot)$ nor $\eta(\cdot,\cdot)$ are continuous functions of their arguments.
\end{remark}

In the theory of the joint/lower spectral radius, the Berger-Wang theorem~\cite{BerWang:LAA92,Czornik:LAA05,Gurv:LAA95,Theys:PhD05} plays a significant role. This theorem makes it possible to express the joint and lower spectral radii by means of equalities~\eqref{E-GSR} and~\eqref{E-LSR1}, respectively. In this connection, the following problems arise.

\begin{problem}\label{Q:2}
Do the following equalities hold
\begin{equation}\label{E:muetaeq2}
\begin{aligned}
\mu(\setA,\setB)&=\Hat{\mu}(\setA,\setB), &\eta(\setA,\setB)&=\Hat{\eta}(\setA,\setB),\\  \mu(\setA,\setB)&=\Check{\mu}(\setA,\setB), &\eta(\setA,\setB)&=\Check{\eta}(\setA,\setB)
\end{aligned}
\end{equation}
(if at least some are true), i.e. are the analogues of the Berger-Wang theorem valid for the corresponding minimax quantities?
\end{problem}

\begin{problem}\label{Q:3}
If the answer to the previous problem is generally negative, then for which sets of matrices $\setA$ and $\setB$ all or some of equalities~\eqref{E:muetaeq2} hold?
\end{problem}

Although Theorem~\ref{T:Hsets} describes one of the cases in which equalities~\eqref{E:muetaeqalities} are true, nevertheless the following problem remains relevant.

\begin{problem}\label{Q:5}
Since according to Example~\ref{Ex:2} in the general case
\[
\mu(\setA,\setB)\neq\eta(\setA,\setB),\quad \Hat{\mu}(\setA,\setB)\neq\Hat{\eta}(\setA,\setB),\quad \Check{\mu}(\setA,\setB)\neq\Check{\eta}(\setA,\setB),
\]
then for which sets of matrices $\setA$ and $\setB$ all or some of equalities~\eqref{E:muetaeqalities} hold?
\end{problem}

\section{Proofs}\label{S:aux}
\subsection{Proof of Lemma~\ref{L:MJSR-correctness}}\label{S:proof-MJSR-correctness}

We need the following auxiliary assertion.
\begin{lemma}\label{L1} Let $ X $, $ Y $, $ U $ and $ V $ be compacts in some topological spaces. Then the following assertions hold.
\begin{enumerate}[\rm(i)]
\item Let $f(x)$, $g(y)$ and $h(x,y)$ be continuous nonnegative functions on the sets $X$, $Y$ and $X\times{}Y$, respectively, such that for any $x\in X$, $y\in Y$ the following inequality holds:
\begin{equation}\label{E:1}
h(x,y)\le f(x)\, g(y).
\end{equation}
Then
\begin{equation}\label{E:2}
\begin{aligned}
\max_{x\in X, y\in Y} h(x,y)&\le \Bigl(\max_{x\in X} f(x)\Bigr)\Bigl(\max_{y\in Y} g(y)\Bigr),\\
\min_{x\in X, y\in Y} h(x,y)&\le \Bigl(\min_{x\in X} f(x)\Bigr)\Bigl(\min_{y\in Y} g(y)\Bigr).
\end{aligned}
\end{equation}

\item Let $F(x,u)$, $G(y,v)$ and $H(x,y,u,v)$ be continuous nonnegative functions on the sets $X\times{}U$, $Y\times{}V$ and $X\times{}Y\times{}U\times{}V$, respectively, such that for any $x\in X$, $y\in Y$, $u\in U$, $v\in V$ the following inequality holds:
\begin{equation}\label{E:1ii}
H(x,y,u,v)\le F(x,u)\, G(y,v).
\end{equation}
Then
\begin{equation}\label{E:2ii}
\begin{aligned}
\min_{u\in U, v\in V} \max_{x\in X, y\in Y} H(x,y,u,v)&\le \Bigl(\min_{u\in U} \max_{x\in X} F(x,u)\Bigr)\Bigl(\min_{v\in V} \max_{y\in Y} G(y,v)\Bigr),\\
\max_{x\in X, y\in Y} \min_{u\in U, v\in V} H(x,y,u,v)&\le \Bigl(\max_{x\in X} \min_{u\in U} F(x,u)\Bigr)\Bigl(\max_{y\in Y} \min_{v\in V} G(y,v)\Bigr).
\end{aligned}
\end{equation}
\end{enumerate}
\end{lemma}

\begin{proof}
(i) The first of the inequalities in~\eqref{E:2} follows from~\eqref{E:1} and the next obvious estimate:
\[
f(x)\, g(y)\le \Bigl(\max_{x\in X} f(x)\Bigr)\Bigl(\max_{y\in Y} g(y)\Bigr),\qquad \forall~x\in X,~y\in Y.
\]

To prove the second inequality in~\eqref{E:2}, denote by $x_{0}\in X$ and $y_{0}\in Y$ the points at which the minima of the functions $f(x)$ and $g(y)$,  respectively, are attained. Then, obviously,
\[
f(x_{0})\, g(y_{0})\le \Bigl(\min_{x\in X} f(x)\Bigr)\Bigl(\min_{y\in Y} g(y)\Bigr)
\]
In this case, by virtue of~\eqref{E:1}
\[
\min_{x\in X, y\in Y} h(x,y)\le h(x_{0},y_{0})\le f(x_{0})\, g(y_{0}),
\]
whence the second of the inequalities in~\eqref{E:2} follows. Assertion (i) is proved.

(ii) Let us prove the first of inequalities~\eqref{E:2ii}. Fix $u\in U$, $v\in V$, and calculate the maxima of the functions $F(x,u)$, $G(y,v)$ and $H(x,y,u,v)$ over all $x\in X$, $y\in Y$. By virtue of the already proved assertion~(i), we obtain that
\[
\max_{x\in X, y\in Y} H(x,y,u,v)\le \Bigl(\max_{x\in X} F(x,u)\Bigr)\Bigl(\max_{y\in Y} G(y,v)\Bigr).
\]
Taking now in this inequality the minimum of both sides in $u\in U$, $v\in V$, we, again by assertion~(i), obtain the first of inequalities~\eqref{E:2ii}.

The second of inequalities~\eqref{E:2ii} can be proved similarly.
\end{proof}

Now we pass to the proof of Lemma~\ref{L:MJSR-correctness}.

\begin{proof}[Proof of Lemma~\ref{L:MJSR-correctness}]
We recall that a nonnegative function of the natural argument $\nu(n)$ is said to be \emph{submultiplicative} if for all $m,n\ge0$ the inequality $\nu(m+n)\le\nu(m)\nu(n)$ holds. Let us show that the functions $\mu_{n}(\setA,\setB)$ and $\eta_{n}(\setA,\setB)$, defined by equalities~\eqref{E:defmnen}, are submultiplicative with respect to $n$. Fix integers $m,n\ge1$. Then for any collections of matrices $A_{1},\ldots,A_{m+n}\in\setA$ and $B_{1},\ldots,B_{m+n}\in\setB$ under the assumption of submultiplicativity of the matrix norm $\|\cdot\|$ the inequality
\begin{equation}\label{E:normsubmult}
\|A_{m+n}B_{m+n}\cdots A_{1}B_{1}\|
\le \|A_{m+n}B_{m+n}\cdots A_{n+1}B_{n+1}\|\cdot \|A_{n}B_{n}\cdots A_{1}B_{1}\|
\end{equation}
takes place. Introducing the variables
\begin{alignat*}{2}
x&=\{A_{1},\ldots,A_{n}\},\quad& y&=\{A_{n+1},\ldots,A_{n+m}\},\\
u&=\{B_{1},\ldots,B_{n}\},& v&=\{B_{n+1},\ldots,B_{n+m}\}
\end{alignat*}
and the functions
\begin{align*}
F(x,u)&=\|A_{n}B_{n}\cdots A_{1}B_{1}\|,\\
G(y,v)&=\|A_{n+m}B_{n+m}\cdots A_{n+1}B_{n+1}\|,\\ H(x,y,u,v)&=\|A_{m+n}B_{m+n}\cdots A_{1}B_{1}\|,
\end{align*}
from~\eqref{E:normsubmult} we get that the functions $F,G$ and $H$ satisfy condition~\eqref{E:1ii}. Then by virtue of assertion~(ii) of Lemma~\ref{L1}, the following inequalities hold:
\[
\mu_{m+n}(\setA,\setB)\le \mu_{m}(\setA,\setB) \mu_{n}(\setA,\setB),\quad
\eta_{m+n}(\setA,\setB)\le \eta_{m}(\setA,\setB) \eta_{n}(\setA,\setB),
\]
which mean that the functions $\mu_{n}(\setA,\setB)$ and $\eta_{n}(\setA,\setB)$ are submultiplicative with respect to $n$. In this case the existence of the limits in~\eqref{E:minmaxrad} and the validity of equalities~\eqref{E:MJSR-inf} follows from the well-known Fekete lemma~\cite{Fekete:MZ23}.

To prove the independence of the limits in~\eqref{E:minmaxrad} from the choice of the norm, we first note that any two norms $\|\cdot\|_{1}$ and $\|\cdot\|_{2}$ on the set of $(N\times N)$-matrices are equivalent to each other, i.e. for them there are constants $c,C>0 $ such that
\[
c\|X\|_{1}\le \|X\|_{2}\le C\|X\|_{1}
\]
for any $(N\times N)$-matrix $X$. In this case, for each matrix product $A_{n}B_{n}\cdots A_{1}B_{1}$ the following inequalities hold:
\[
c^{\frac{1}{n}}\|A_{n}B_{n}\cdots A_{1}B_{1}\|_{1}^{\frac{1}{n}}\le \|A_{n}B_{n}\cdots A_{1}B_{1}\|_{2}^{\frac{1}{n}}\le C^{\frac{1}{n}}\|A_{n}B_{n}\cdots A_{1}B_{1}\|_{1}^{\frac{1}{n}}.
\]
Here, obviously, $c^{\frac{1}{n}}, C^{\frac{1}{n}}\to 1$ as $n\to\infty$, from which it follows that the corresponding limits in~\eqref{E:minmaxrad}, computed for the norms $\|\cdot\|_{1}$ and $\|\cdot\|_{2}$, actually coincide. This completes the proof of the independence of the limits in~\eqref{E:minmaxrad} from the choice of the norm.
\end{proof}

\subsection{Proofs of Theorems~\ref{T:pointwizestab} and~\ref{T:strongstab}}\label{S:proofstab}

For the proof we need a number of auxiliary assertions.

\begin{lemma}[{see~\cite[Theorem~1]{Koz:Arxiv17-1}}]\label{L:pathstab}
Let $\setA$ and $\setB$ be finite sets of matrices of dimension $N\times{}M$ and $M\times{}N$, respectively, and $\|\cdot\|$ be a norm on the space of matrices of dimension $N\times{}N$. If the corresponding system $\systemAB$ is path-dependent stabilizable, then there exist constants $C>0$ and $\lambda\in(0,1)$ such that for any sequence of matrices $\{A_{n}\in\setA\}$ there is a sequence of matrices $\{B_{n}\in\setB\}$ for which
\[
\|A_{n}B_{n}\cdots A_{1}B_{1}\|\le C\lambda^{n},\qquad n=1,2,\ldots\,.
\]
\end{lemma}

\begin{lemma}[{see~\cite[Theorem~2]{Koz:Arxiv17-1}}]\label{L:unistab}
Let $\setA$ and $\setB$ be finite sets of matrices of dimension $N\times{}M$ and $M\times{}N$, respectively, and $\|\cdot\|$ be a norm on the space of matrices of dimension $N\times{}N$. Suppose that the corresponding system $\systemAB$ is path-independent periodically stabilizable and $\{\Bar{B}_{n}\in\setB\}$ is a sequence of matrices that realizes its universal periodic stabilization. Then there exist constants $C>0$ and $\lambda\in(0,1)$ such that the inequalities
\[
\|A_{n}\Bar{B}_{n}\cdots A_{1}\Bar{B}_{1}\|\le C\lambda^{n},\qquad n=1,2,\ldots\,,
\]
hold for any sequence of matrices $\{A_{n}\in\setA\}$.
\end{lemma}

\begin{lemma}\label{L:mJSRlessone}
Let $\setA$ and $\setB$ be finite sets of matrices of dimension $N\times{}M$ and $M\times{}N$, respectively, and $\|\cdot\|$ be a norm on the space of matrices of dimension $N\times{}N$. Then the following assertions hold.
\begin{enumerate}[\rm(i)]
\item If $\mu(\setA,\setB)<1$, then there exist a constant $\sigma\in(0,1)$ and a positive integer $k$ such that
    \begin{equation}\label{E:mucrit}
     \forall~A_{1},\ldots,A_{k}\in\setA\quad \exists~ B_{1},\ldots,B_{k}\in\setB:\quad
     \|A_{k}B_{k}\cdots A_{1}B_{1}\|\le\sigma.
    \end{equation}
    If there exist a constant $\sigma\in(0,1)$ and a positive integer $k$ for which condition~\eqref{E:mucrit} holds in some submultiplicative matrix norm $\|\cdot\|$, then $\mu(\setA,\setB)<1$.
\item If $\eta(\setA,\setB)<1$, then there exist a constant $\sigma\in(0,1)$ and a positive integer $k$ such that
    \begin{equation}\label{E:etacrit}
     \exists~ \Bar{B}_{1},\ldots,\Bar{B}_{k}\in\setB:\quad
     \|A_{k}\Bar{B}_{k}\cdots A_{1}\Bar{B}_{1}\|\le\sigma\quad \forall~A_{1},\ldots,A_{k}\in\setA.
    \end{equation}
    If there exist a constant $\sigma\in(0,1)$ and a positive integer $k$ for which condition~\eqref{E:etacrit} holds in some submultiplicative matrix norm $\|\cdot\|$, then $\eta(\setA,\setB)<1$.
\end{enumerate}
\end{lemma}

\begin{proof}
If $\mu(\setA,\setB)<1$, then by the definition~\eqref{E:minmaxrad} there can be found a constant $\sigma\in(0,1)$ and a positive integer $k$ such that $\mu_{k}(\setA,\setB)\le\sigma$. The last condition, by the definition~\eqref{E:defmnen} of the value $\mu_{k}(\setA,\setB)$, is just written as~\eqref{E:mucrit}.

Now let $\|\cdot\|$ be a submultiplicative matrix norm for which condition~\eqref{E:mucrit} holds for some $\sigma\in(0,1)$ and natural $k$. Then for any collection of matrices $A_{1},\ldots,A_{k}\in\setA$ there is a collection of matrices $B_{1},\ldots,B_{k}\in\setB$ such that $\|A_{k}B_{k}\cdots A_{1}B_{1}\|\le\sigma$. Hence, for any set of matrices $A_{1},\ldots,A_{k}\in\setA$, the estimate
\[
\min_{B_{i}\in\setB} \|A_{k}B_{k}\cdots A_{1}B_{1}\|\le\sigma
\]
holds. Therefore,
\[
\mu_{k}(\setA,\setB)=\max_{A_{i}\in\setA}\min_{B_{i}\in\setB} \|A_{k}B_{k}\cdots A_{1}B_{1}\|\le\sigma
\]
and, by virtue of relations~\eqref{E:MJSR-inf}, from Lemma~\ref{L:MJSR-correctness} we obtain that  $\mu(\setA,\setB)\le\sigma^{\frac{1}{k}}<1$. Assertion~(i) is proved.

The proof of assertion~(ii) is similar.
\end{proof}

\begin{proof}[Proof of Theorem~\ref{T:pointwizestab}]
Suppose that the system $\systemAB$ is path-dependent stabilizable in the class of all perturbations $\setA$ and controls $\setB$. Fix some number $\sigma\in(0,1)$. For it, by Lemma~\ref{L:pathstab}, a natural $k$ can be found such that for any collection of matrices $A_{1},\ldots,A_{k}\in\setA$ there is a collection of matrices  $B_{1},\ldots,B_{k}\in\setB$ satisfying
\[
\|A_{k}B_{k}\cdots A_{1}B_{1}\|\le \sigma<1.
\]
Hence, by assertion~(i) of Lemma~\ref{L:mJSRlessone}, we obtain that $\mu(\setA,\setB)<1$.

Now let $\mu(\setA,\setB)<1$. Then by assertion~(i) of Lemma~\ref{L:mJSRlessone} there is a constant $\sigma\in(0,1)$ and a natural $k$ such that condition~\eqref{E:mucrit} is satisfied. We will show that the fulfillment of this condition implies the path-dependent stabilizability of the system $\systemAB$ in the class of all perturbations $\setA$ and controls $\setB$.

Given an arbitrary sequence $\{A_{n}\in\setA\}$, by condition~\eqref{E:mucrit} there exists for it a collection of matrices $B_{1},\ldots,B_{k}$ such that
\[
\|A_{k}B_{k}\cdots A_{1}B_{1}\|\le\sigma.
\]

Next, consider the sequence of matrices $\{A_{n}\in\setA,~n\ge k+1\}$ (the `tail' of the sequence $\{A_{n}\in\setA\}$ starting with index $k+1$). Again, because of condition~\eqref{E:mucrit}, there is a collection of matrices $B_{k+1},\ldots,B_{2k}$ such that
\[
\|A_{2k}B_{2k}\cdots A_{k+1}B_{k+1}\|\le\sigma.
\]

We continue in the same way to construct for each $m=3,4,\ldots$ the collections of matrices $B_{k(m-1)+1},\ldots,B_{km}$ for which
\[
\|A_{km}B_{km}\cdots A_{k(m-1)+1}B_{k(m-1)+1}\|\le\sigma.
\]

It is easy to see that the constructed sequence of matrices $\{B_{n}\}$ satisfies the inequalities
\[
\|A_{km}B_{km}\cdots A_{1}B_{1}\|\le\sigma^{m},\qquad m=1,2,\ldots\,,
\]
whence by the boundedness of the norms of all matrices $AB$, where $A\in\setA$ and $B\in\setB$ (recall that the sets of matrices $\setA$ and $\setB$ are finite), the matrix products $A_{n}B_{n}\cdots A_{1}B_{1}$ converge to zero. The theorem is proved.
\end{proof}

The proof of Theorem~\ref{T:strongstab} almost literally repeats the above proof of Theorem~\ref{T:pointwizestab}. Nevertheless, for the sake of completeness, we give it, too.

\begin{proof}[Proof of Theorem~\ref{T:strongstab}]
Suppose that the system $\systemAB$ is path-independent periodically stabilizable in the class of all perturbations $\setA$. Fix some number $\sigma\in(0,1)$. For it, by Lemma~\ref{L:unistab}, we can find a natural $k$ and a sequence of matrices $\{\Bar{B}_{n}\in\setB\}$ such that for any collection of matrices $A_{1},\ldots,A_{k}\in\setA$ the inequalities
\[
\|A_{k}\Bar{B}_{k}\cdots A_{1}\Bar{B}_{1}\|\le\sigma<1
\]
hold. From here, by assertion~(ii) of Lemma~\ref{L:mJSRlessone}, we obtain that $\eta(\setA,\setB)<1$.

Now let $\eta(\setA,\setB)<1$. Then by assertion~(ii) of Lemma~\ref{L:mJSRlessone} there exists a constant $\sigma\in(0,1)$, a natural $k$ and a collection of matrices $\Bar{B}_{1},\ldots,\Bar{B}_{k}\in\setA$, for which condition~\eqref{E:etacrit} is satisfied. Extend the collection of matrices $\Bar{B}_{1},\ldots,\Bar{B}_{k}\in\setA$ by periodicity to the infinite $k$-periodic sequence $\{\Bar{B}_{n}\in\setB\}$. We show that in this case the system $\systemAB$ will be path-independent periodically stabilizable by the sequence $\{\Bar{B}_{n}\in\setB\}$ in the class of all perturbations $\setA$.

Take an arbitrary sequence $\{A_{n}\in\setA\}$. Then, for each $m=1,2,\ldots$\,, by virtue of condition~\eqref{E:etacrit} and $k$-periodicity of the sequence $\{\Bar{B}_{n}\in\setB\}$ the relations
\[
\|A_{km}\Bar{B}_{km}\cdots A_{k(m-1)+1}\Bar{B}_{k(m-1)+1}\|=\|A_{km}\Bar{B}_{k}\cdots A_{k(m-1)+1}\Bar{B}_{1}\|\le\sigma
\]
will take place. Consequently,
\[
\|A_{km}\Bar{B}_{km}\cdots A_{1}\Bar{B}_{1}\|\le\sigma^{m},\qquad m=1,2,\ldots\,,
\]
whence by the boundedness of the norms of all matrices $AB$, where  $A\in\setA$ and $B\in\setB$  (recall that the sets of matrices $\setA$ and $\setB$  are finite), the matrix products $A_{n}\Bar{B}_{n}\cdots A_{1}\Bar{B}_{1}$ are convergent to zero.

The theorem is proved.
\end{proof}

\subsection{Proof of Theorem~\ref{T:Hsets}}\label{S:proofTHsets}

In the following discussions, the minimax equality
\begin{equation}\label{E-minimax}
\min_{B\in\setB}\max_{A\in\setA}\rho(AB)=
\max_{A\in\setA}\min_{B\in\setB}\rho(AB),
\end{equation}
which is valid for any compact sets of matrices $\setA\in\overline{\setH}(N,M)$ and $\setB\in\overline{\setH}(M,N)$, plays the key role. This equality is obtained from the following equality proved in~\cite[Theorem~3.3]{Koz:LMA17}:
\[
\min_{A\in\setA}\max_{B\in\setB}\rho(AB)=
\max_{B\in\setB}\min_{A\in\setA}\rho(AB),
\]
if to interchange the variables $A$  and $B$ (and the sets $\setA$ and $\setB$, respectively) in the latter and notice that for any matrices $A$ and  $B$ (rectangular, in the general case) the equality $\rho(AB)=\rho(BA)$ holds.

In the theory of functions, one of the fundamental criteria for the feasibility of a minimax equality is the so-called \emph{saddle point principle}, see~\cite[Section~13.4]{von1947theory}, which, in relation to the situation we are considering, states that the minimax equality~\eqref{E-minimax} is satisfied if and only if there are matrices $\tilde{A}\in\setA$ and $\tilde{B}\in\setB$ such that
\begin{equation}\label{E:saddleAB}
\rho(A\tilde{B})\le\rho(\tilde{A}\tilde{B})\le \rho(\tilde{A}B),
\end{equation}
for all $A\in\setA$ and $B\in\setB$.

Also, an important property of $\setH$-sets of matrices is that for them the joint and lower spectral radii can be calculated constructively. In particular, as shown in~\cite[Theorem~3]{Koz:LAA16}, for any compact set of matrices $\setA\in\overline{\setH}(N,N)$ the equalities
\begin{equation}\label{E:exact0}
\max_{A_{i}\in\setA} \rho{(A_{n}\cdots A_{1})}^{1/n}= \max_{A\in\setA}\rho(A),\qquad\min_{A_{i}\in\setA} \rho{(A_{n}\cdots A_{1})}^{1/n}= \min_{A\in\setA}\rho(A)
\end{equation}
hold  for each $n\ge 1$, from which it follows by virtue of~\eqref{E-LSR} and~\eqref{E-GSR} that
\begin{equation}\label{E-exact}
\rho(\setA)=\max_{A\in\setA}\rho(A),\qquad \check{\rho}(\setA)=\min_{A\in\setA}\rho(A).
\end{equation}

The following lemma, which is a natural generalization of Theorem~3 from~\cite{Koz:LAA16}, shows that for $\setH$-sets of matrices, not only the joint and lower spectral radii can be constructively calculated, but also the minimax quantities $\Hat{\mu}(\setA,\setB)$, $\Hat{\eta}(\setA,\setB)$, $\Check{\mu}(\setA,\setB)$ and $\Check{\eta}(\setA,\setB)$.

\begin{lemma}\label{L:Hsetminimax}
Let $\setA$, $\setB$ be compact $\setH$-sets of positive matrices of dimension  $M\times N$ and $N\times M$, respectively. Then
\begin{equation}\label{E:bmneqben}
\Bar{\mu}_{n}(\setA,\setB)=\Bar{\eta}_{n}(\setA,\setB) =\rho(\tilde{A}\tilde{B})^{n},\quad n\ge 1,
\end{equation}
where the quantities $\Bar{\mu}_{n}(\setA,\setB)$ and $\Bar{\eta}_{n}(\setA,\setB)$ are defined by equalities~\eqref{E:defbmbe}, and therefore, by virtue of the definitions~\eqref{E:minmaxhigh} and~\eqref{E:minmaxlow},
\[
\Hat{\mu}(\setA,\setB)=\Check{\mu}(\setA,\setB) =\Hat{\eta}(\setA,\setB)=\Check{\eta}(\setA,\setB)= \rho(\tilde{A}\tilde{B}).
\]
\end{lemma}

\begin{proof}
We will use the idea of a proof of Theorem~13 from~\cite{ACDDHK:STACS15}, which is close in meaning. Fix $n\ge 1$. By the definition~\eqref{E:defbmbe} of the quantity $\Bar{\mu}_{n}(\setA,\setB)$ we have
\begin{equation}\label{E-etanestim}
\begin{aligned}
\Bar{\mu}_{n}(\setA,\setB)&=\max_{A_{i}\in\setA} \min_{B_{i}\in\setB}\rho(A_{n}B_{n}\cdots A_{1}B_{1})\\
&\ge
\min_{B_{i}\in\setB}\rho(\tilde{A}B_{n}\cdots \tilde{A}B_{1})=\min_{\tilde{B}_{i}\in\tilde{A}\setB}\rho(\tilde{B}_{n}\cdots \tilde{B}_{1}),
\end{aligned}
\end{equation}
where  $\tilde{A}\setB:=\{\tilde{B}:~\tilde{B}=\tilde{A}B,~B\in\setB\}$.

We now note that, as follows from~\cite[Theorem~2]{Koz:LAA16}, the set of matrices $\tilde{A}\setB$ is an $\setH$-set, since the one-element set of matrices $\{\tilde{A}\}$ and the set of matrices $\setB$ are both $\setH$-sets. Therefore, by the second of equalities~\eqref{E:exact0} and the definition~\eqref{E:saddleAB} of matrices  $\tilde{A}$ and $\tilde{B}$,
\[
\min_{\tilde{B}_{i}\in\tilde{A}\setB}\rho(\tilde{B}_{n}\cdots \tilde{B}_{1})=\min_{\tilde{B}\in\tilde{A}\setB} \rho(\tilde{B})^{n}=\min_{B\in\setB} \rho(\tilde{A}B)^{n}\ge\rho(\tilde{A}\tilde{B})^{n},
\]
from which by virtue of~\eqref{E-etanestim}
\begin{equation}\label{E:etanAB}
\Bar{\mu}_{n}(\setA,\setB)\ge\rho(\tilde{A}\tilde{B})^{n}.
\end{equation}

Similarly, we estimate the value of $\Bar{\eta}_{n}(\setA,\setB)$. To this end, using the definition~\eqref{E:defbmbe}, we write the chain of relations
\begin{equation}\label{E-munestim}
\begin{aligned}
\Bar{\eta}_{n}(\setA,\setB)&= \min_{B_{i}\in\setB}\max_{A_{i}\in\setA} \rho(A_{n}B_{n}\cdots A_{1}B_{1})\\
&\le
\max_{A_{i}\in\setA}\rho(A_{n}\tilde{B}\cdots A_{1}\tilde{B})= \max_{\tilde{A}_{i}\in\setA\tilde{B}} \rho(\tilde{A}_{n}\cdots \tilde{A}_{1}),
\end{aligned}
\end{equation}
where $\setA\tilde{B}:=\{\tilde{A}:~\tilde{A}=A\tilde{B},~A\in\setA\}$.

Again, by virtue of~\cite[Theorem~2]{Koz:LAA16}, the set of matrices $\setA\tilde{B}$ is an $\setH$-set, since the one-element set of matrices $\{\tilde{B}\}$ and the set of matrices $\setA$ are both $\setH$-sets. Therefore, by the first of equalities~\eqref{E:exact0} and the definition~\eqref{E:saddleAB} of matrices~$\tilde{A}$ and~$\tilde{B}$,
\[
\max_{\tilde{A}_{i}\in\setA\tilde{B}} \rho(\tilde{A}_{n}\cdots \tilde{A}_{1})=\max_{\tilde{A}\in\setA\tilde{B}} \rho(\tilde{A})^{n}=\max_{A\in\setA} \rho(A\tilde{B})^{n}\le\rho(\tilde{A}\tilde{B})^{n},
\]
from which by virtue of~\eqref{E-munestim}
\begin{equation}\label{E:munAB}
\Bar{\eta}_{n}(\setA,\setB)\le\rho(\tilde{A}\tilde{B})^{n}.
\end{equation}

Comparing inequalities~\eqref{E:etanAB} and~\eqref{E:munAB}, we obtain that $\Bar{\mu}_{n}(\setA,\setB)\ge\rho(\tilde{A}\tilde{B})^{n}\ge\Bar{\eta}_{n}(\setA,\setB)$. On the other hand, since the maximin of a function does not exceed its minimax, we find from the definition~\eqref{E:defbmbe} that $\Bar{\mu}_{n}(\setA,\setB)\le\Bar{\eta}_{n}(\setA,\setB)$, which implies equality~\eqref{E:bmneqben}. The lemma is proved.
\end{proof}

Now we can proceed to the proof of Theorem~\ref{T:Hsets}.

\begin{proof}[Proof of Theorem~\ref{T:Hsets}]

The equality of the values $\Hat{\mu}(\setA,\setB)$, $\Check{\mu}(\setA,\setB)$, $\Hat{\eta}(\setA,\setB)$ and $\Check{\eta}(\setA,\setB)$ (and their equality to $\rho(\tilde{A}\tilde{B})$) is established in Lemma~\ref{L:Hsetminimax}. Therefore, it remains to prove
\begin{equation}\label{E:mueqeta}
\mu(\setA,\setB)=\eta(\setA,\setB)=\rho(\tilde{A}\tilde{B}).
\end{equation}
Unfortunately, as follows from Example~\ref{Ex:2}, in the general case $\mu_{n}(\setA,\setB)\neq\eta_{n}(\setA,\setB)$, and therefore we can not directly use the scheme of proving Lemma~\ref{L:Hsetminimax} to prove equalities~\eqref{E:mueqeta}. Nevertheless, the idea of proving Lemma~\ref{L:Hsetminimax} still turns out to be workable after small changes.

Let $\|\cdot\|$ be the matrix norm on the space of matrices of dimension $N\times{}N$ generated by some vector norm $\|\cdot\|$ on $\mathbb{R}^{N}$. In this case, by the definition~\eqref{E:minmaxrad} of the quantity $\mu(\setA,\setB)$, we have
\begin{equation}\label{E-etaestim}
\begin{aligned}
\mu(\setA,\setB)&=\lim_{n\to\infty}\max_{A_{i}\in\setA} \min_{B_{i}\in\setB}\|A_{n}B_{n}\cdots A_{1}B_{1}\|^{\frac{1}{n}}
\\
&\ge
\lim_{n\to\infty}\min_{B_{i}\in\setB}\|\tilde{A}B_{n}\cdots \tilde{A}B_{1}\|^{\frac{1}{n}}=\lim_{n\to\infty}\min_{\tilde{B}_{i}\in\tilde{A}\setB}\|\tilde{B}_{n}\cdots \tilde{B}_{1}\|^{\frac{1}{n}}\\
&=\Check{\rho}(\tilde{A}\setB),
\end{aligned}
\end{equation}
where $\tilde{A}\setB:=\{\tilde{B}:~\tilde{B}=\tilde{A}B,~B\in\setB\}$.

As was mentioned in the proof of Lemma~\ref{L:Hsetminimax}, the set of matrices $\tilde{A}\setB$ is an $\setH$-set. Therefore, by the second of equalities~\eqref{E-exact} and inequalities~\eqref{E:saddleAB},
\[
\Check{\rho}(\tilde{A}\setB)=\min_{\tilde{B}\in\tilde{A}\setB} \rho(\tilde{B})=\min_{B\in\setB} \rho(\tilde{A}B)\ge\rho(\tilde{A}\tilde{B}),
\]
from which by virtue of~\eqref{E-etaestim}
\begin{equation}\label{E:etaAB}
\mu(\setA,\setB)\ge\rho(\tilde{A}\tilde{B}).
\end{equation}

Similarly, we estimate the value of $\eta(\setA,\setB)$. To this end, using the definition~\eqref{E:minmaxrad}, we write the chain of relations
\begin{equation}\label{E-muestim}
\begin{aligned}
\eta(\setA,\setB)&=\lim_{n\to\infty} \min_{B_{i}\in\setB}\max_{A_{i}\in\setA} \|A_{n}B_{n}\cdots A_{1}B_{1}\|^{\frac{1}{n}}\\
&\le
\lim_{n\to\infty} \max_{A_{i}\in\setA} \|A_{n}\tilde{B}\cdots A_{1}\tilde{B}\|^{\frac{1}{n}}=\lim_{n\to\infty} \max_{\tilde{A}_{i}\in\setA\tilde{B}} \|\tilde{A}_{n}\cdots \tilde{A}_{1}\|^{\frac{1}{n}}\\
&=\rho(\setA\tilde{B}),
\end{aligned}
\end{equation}
where $\setA\tilde{B}:=\{\tilde{A}:~\tilde{A}=A\tilde{B},~A\in\setA\}$.

Here again, as was mentioned in the proof of Lemma~\ref{L:Hsetminimax}, the set of matrices $\setA\tilde{B}$ is an $\setH$-set. Therefore, by the first of equalities~\eqref{E-exact} and inequalities~\eqref{E:saddleAB}
\[
\rho(\setA\tilde{B})=\max_{\tilde{A}\in\setA\tilde{B}} \rho(\tilde{A})=\max_{A\in\setA} \rho(A\tilde{B})\le\rho(\tilde{A}\tilde{B}),
\]
from which by virtue of~\eqref{E-muestim}
\begin{equation}\label{E:muAB}
\eta(\setA,\setB)\le\rho(\tilde{A}\tilde{B}).
\end{equation}

Comparing inequalitie~\eqref{E:etaAB} and~\eqref{E:muAB}, we get that $\mu(\setA,\setB)\ge\rho(\tilde{A}\tilde{B})\ge\eta(\setA,\setB)$. But on the other hand, since the minimax of a function does not exceed its maximin, from the definitions~\eqref{E:minmaxrad} we get that $\mu(\setA,\setB)\le\eta(\setA,\setB)$, from which equality~\eqref{E:mueqeta} follows.

The theorem is proved.
\end{proof}

\begin{remark}
In Theorem~\ref{T:Hsets}, instead of the sets $\setA\in\setH(N,M)$ and $\setB\in\setH(M,N)$, one can take the sets $\Tilde{\setA}$ and $\Tilde{\setB}$ satisfying inclusions
\begin{equation}\label{E:ABco}
\setA\subseteq\Tilde{\setA}\subseteq\co(\setA),\quad \setB\subseteq\Tilde{\setB}\subseteq\co(\setB),
\end{equation}
where $\setA\in\overline{\setH}(N,M)$ and $\setB\in\overline{\setH}(M,N)$, and the symbol $\co(\cdot)$ denotes the convex hull of a set. The validity of this remark follows from the fact that all the statements used in the proof of Theorem~\ref{T:Hsets} are proved in~\cite{Koz:LMA17} namely for the sets $\Tilde{\setA}$ and $\Tilde{\setB}$ satisfying the inclusions~\eqref{E:ABco}.
\end{remark}

\section*{Acknowledgments} The work was carried out at the Kotel'nikov Institute of Radio-engineering and Electronics, Russian Academy
of Sciences, and was supported by the Russian Science Foundation, Project number 16-11-00063.


\bibliographystyle{elsarticle-num-nourl}
\bibliography{MMJSR}

\end{document}